\documentclass[review]{elsarticle}
\usepackage{lineno,hyperref}
\journal{Journal of \LaTeX\ Templates}
\usepackage[centertags]{amsmath}
\usepackage{epsfig}
\usepackage{amsfonts,graphicx}
\usepackage{amssymb}
\usepackage{amsthm}
\usepackage{verbatim}
\usepackage{newlfont}
\theoremstyle{Corollary}
\newtheorem{cor}{Corollary}[section]
\theoremstyle{example}

\theoremstyle{lemma}
\newtheorem{lem}{Lemma}[section]
\theoremstyle{definition}
\newtheorem{defn}{Definition}[section]
\theoremstyle{theorem}
\newtheorem{rem}{Remark}[section]
\theoremstyle{theorem}
\newtheorem{thm}{Theorem}[section]
\theoremstyle{proposition}
\newtheorem{prop}{Proposition}[section]
\begin{document}

\begin{frontmatter}

\title{J-class  sequences of linear operators}
%\tnotetext[mytitlenote]{}

\author{M.R.  Azimi}
\address{Department of Mathematics, Faculty of Science, University of Maragheh, 55181-83111 Maragheh, Iran.}

%\cortext[mycorrespondingauthor]{M.R.  Azimi}
\ead{mhr.azimi@maragheh.ac.ir}

\begin{abstract}
In this paper we first introduce  the extended limit set
$J_{\{T^n\}}(x)$ for a sequence of bounded linear operators
$\{T_n\}_{n=1}^{\infty}$ on a separable Banach space $X$ . Then we
study the dynamics of sequence of linear operators by using the
extended limit set. It is shown that the extended limit set is
strongly related to the topologically transitive of a sequence of
linear operators. Finally we show that a sequence of operators
$\{T_n\}_{n=1}^{\infty}\subseteq \mathcal{B}(X)$ is hypercyclic if
and only if there exists a cyclic vector $x\in X$ such that
$J_{\{T^n\}}(x)=X$.
\end{abstract}


\end{frontmatter}

%\linenumbers
%%%%%%%%%%%%%%%%%%%%%%%%%%%%%%%%%%%%%%%%%%%%%%%%%%%%%%%%%%%%%%%%%%%%%%%%%%%%
%%%%%%%%%%%%%%%%%%%%%%%%%%%%%%%%%%%%%%%%%%%%%%%%%%%%%%%%%%%%%%%%%%%%%%%%%%%%
\section{Introduction and Preliminaries}\label{section1}
Let $X$ be a separable Banach space over the field of complex
numbers $\mathbb{C}$ and let $\mathcal{B}(X)$ be the algebra of all
bounded linear operators on $X$. An operator $T\in \mathcal{B}(X)$
is called \emph{hypercyclic} if there exists a vector $x\in X$ whose
orbit under $T$, $Orb(T, x)=\{T^n x: n\in \mathbb{N}\}$, is dense in
$X$. Such a vector is called \emph{hypercyclic vector} for $T$. An
operator $T\in  \mathcal{B}(X)$ is called \emph{topologically
transitive} if for any pair $U, V$ of nonempty open subsets of $X$,
there exits $n\in \mathbb{N}$ such that
$$T^n( U)\cap V\neq \emptyset.$$
Over last two decades many useful criteria and results have been
presented for the hypercyclicity of linear operators which may find
in \cite{MR2996748}, \cite{MR2533318}, \cite{MR2281650},
\cite{MR2720700} and the references therein.
  It is well known and easy to prove
that an operator $T\in \mathcal{B}(X)$ is hypercyclic if and only if
$T$ is topologically transitive (see \cite{MR2533318}).
\\In recent years the dynamics of linear operators has been widely
studied in terms of their orbits by many authors. The survey
articles \cite{MR2881536}, \cite{MR2015614}, \cite{MR2015614},
\cite{MR2720700}, \cite{MR2000019} and the book \cite{MR2533318} are
good references in this subject. In \cite{MR2881536} G. Costakis and
A. Manoussos have studied the dynamics of linear operators by making
use of extended limit sets instead of their orbits. Especially a
result due to P. Bourdon and N. Feldman (\cite{MR1986898}) has been
generalized by them as follows.
\begin{thm}(\cite{MR2881536}) If $x$ is a cyclic vector for an
operator $T\in \mathcal{B}(X)$ and the set $J_T(x)$ has non-empty
interior then $J_T(x)=X$ and so $T$ is hypercyclic.
\end{thm}
 They have also posed several open
problems which one of them was answered in \cite{MR2826722}. In
particular they have studied $J$-class weighted shifts in
\cite{MR2447911}.  It should be mentioned that H. Salas is the first
mathematician who characterized the hypercyclicity of weighted
shifts in \cite{MR1249890}.  Recently Q. Menet \cite{MR3167484}
generalized the results of  F. Le{\'o}n-Saavedra and V. M{\"u}ller
\cite{MR2261697} and obtained a new criterion for characterizing of
weighted shifts with hypercyclic subspaces.

Let us define the notion of extended limit set. Let $T\in
\mathcal{B}(X)$ and $x\in X$. The \emph{extended limit set} of $x$,
denoted by $J_T(x)$ consists of those vectors $y$ in $X$ if there
exist a strictly increasing sequence of positive integers $\{k_n\}$
and a sequence $\{x_n\} \subseteq X$ such that  $x_n \rightarrow x$
and $T^{k_n}x_n \rightarrow  y$. An operator $T$ is called $J$-class
operator whenever $J_T(x)=X$ for some non-zero vector $x\in X$.

Actually the extended limit set localizes the notion of
hypercyclicity. For more details the reader is referred to
\cite{MR2881536}.

The goal of this paper is to investigate the dynamics of sequences
of linear operators using the  localized hypercyclicity notion. For
this we start by the definition of hypercyclic and topologically
transitive sequences of linear operators.
\begin{defn}
 A sequence of operators
$\{T_n\}_{n=1}^{\infty}\subseteq \mathcal{B}(X)$ is called
\emph{hypercyclic} if there exists a vector $x\in X$ such that the
set $\{T_n x: n\in \mathbb{N}\}$ is dense in $X$. Such a vector is
called hypercyclic vector for the sequence of operators
$\{T_n\}_{n=1}^{\infty}$. We say that an operator $T:X\rightarrow X$
is hypercyclic if the sequence of its iterates
$\{T^n\}_{n=1}^{\infty}$ is hypercyclic.
\end{defn}
\begin{defn}
A sequence of operators $\{T_n\}_{n=1}^{\infty}\subseteq
\mathcal{B}(X)$ is called \emph{topologically transitive} if for any
pair $U, V$ of nonempty open subsets of $X$, there exits $n\in
\mathbb{N}$ such that
$$T_n( U)\cap V\neq \emptyset.$$
\end{defn}
The sequence of linear operators $\{T_n\}_{n=1}^{\infty}$ is
hypercyclic if and only if it is topologically
transitive(\cite{MR1111569}). The  hypercyclicity criterion for
sequences of operators has been studied in \cite{MR1980114}.
 In \cite{MR2261697}, F.
Le{\'o}n-Saavedra and V. M{\"u}ller have studied the  various
conditions and criteria that imply the hypercyclicity of
$\{T_n\}_{n=1}^{\infty}$.

In this paper we first introduce  the extended limit set for a
sequence of bounded linear operators on a separable Banach space $X$
analogous to the single operator case. Then we study the dynamics of
sequence of linear operators by using the extended limit set. It is
turn out that the extended limit set is strongly related to the
hypercyclicity of  a sequence of linear operators. Indeed we show
that a sequence of operators $\{T_n\}_{n=1}^{\infty}\subseteq
\mathcal{B}(X)$ is hypercyclic if and only if there exists a cyclic
vector $x\in X$ such that $J_{\{T^n\}}(x)=X$. Recall that a vector
$x$ is \emph{cyclic} for $\{T_n\}_{n=1}^{\infty}$ if the linear span
of the set $\{T_n x: n\in\mathbb{N} \}$ is dense in $X$.

\begin{defn}
Let $\{T_n\}_{n=1}^{\infty}\subseteq \mathcal{B}(X)$  and $x\in X$.
The \emph{extended limit set} of $x$ under $\{T_n\}_{n=1}^{\infty}$
, denoted by $J_{\{T_n\}}(x)$ is the collection of all  vectors $y$
in $X$ if there exist a strictly increasing sequence of positive
integers $\{k_n\}$ and a sequence $\{x_n\} \subseteq X$ such that
$x_n \rightarrow x$ and $T_{k_n}x_n \rightarrow  y$. A sequence of
bounded linear operators $\{T_n\}_{n=1}^{\infty}$ is called
$J$-class sequence of linear operators whenever $J_{\{T_n\}}(x)=X$
for some non-zero vector $x\in X$.
\end{defn}
By comparing two following definitions
\begin{eqnarray}
J_T(x)=\{y\in X: \mbox{there exist a strictly increasing sequence of
positive integers}\nonumber\\
 \{k_n\}\  \mbox{and a sequence}\  \{x_n\} \subseteq X \ \mbox{such that}\  x_n \rightarrow x \ \mbox{and}\ T^{k_n}x_n \rightarrow
 y\}.\nonumber
\end{eqnarray}
and
\begin{eqnarray}
J_{\{T_n\}}(x)=\{y\in X: \mbox{there exist a strictly increasing
sequence of
positive integers}\nonumber\\
 \{k_n\}\  \mbox{and a sequence}\  \{x_n\} \subseteq X \ \mbox{such that}\  x_n \rightarrow x \ \mbox{and}\ T_{k_n}x_n \rightarrow
 y\}.\nonumber
\end{eqnarray}
it is found out  that $J_T(x)=J_{\{T^n\}}(x)$ for every $x\in X.$

\begin{comment}
\begin{eqnarray}
J_{\{T_n\}}(x)=\{y\in X: \mbox{there exist a strictly increasing
sequence of
positive integers}\nonumber\\
 \{k_n\}\  \mbox{and a sequence}\  \{x_n\} \subseteq X \ \mbox{such that}\  x_n \rightarrow x \ \mbox{and}\ T_{k_n}x_n \rightarrow
 y\}.\nonumber
\end{eqnarray}
$$J_T(x)=J_{\{T^n\}}(x)\ \mbox{for every } x\in X.$$
\end{comment}

\section{Main Results}
In this section we review some basic properties of the extended
limit sets. Afterwards we characterize the hypercyclicity of
$\{T_n\}\subseteq \mathcal{B}(X)$ based upon these sets.
\begin{lem} If $\{T_n\}$ is hypercyclic then $J_{\{T_n\}}(x)=X$ for
every $x\in X.$
\end{lem}
\begin{proof}
Let $x, y \in X$ be arbitrary. Remind that the set of all
hypercyclic vectors for $\{T_n\}$ is dense in $X$, see
\cite{MR1111569}. Hence  there exists a sequence of hypercyclic
vectors $\{x_n\}$ such that $x_n \rightarrow x$. But in view of the
hypercyclicity definition,  the set $\{T_n x: n\in \mathbb{N}\}$
being dense in $X$, one may find a strictly increasing sequence of
positive integers $\{k_n\}$  such that  $T_{k_n}x_n \rightarrow y$.
This means that $y\in J_{\{T_n\}}(x)$ and so $J_{\{T_n\}}(x)=X$.
\end{proof}
\begin{rem}\label{R1}
Note that an equivalent definition of $J_{\{T_n\}}(x)$ can be stated
in the following.
\begin{eqnarray}
J_{\{T_n\}}(x)=\{y\in X: \mbox{for every pair neighborhoods}\ U, V
\ of\ x, y \mbox{  respectively,}\nonumber \\
\mbox{there exists a positive integer} \ n \   \mbox{such that}\ T_n
U \cap V\neq \emptyset \}.\nonumber
 \end{eqnarray}
Observe now that $\{T_n\}$ is topologically transitive if and only
if $J_{\{T_n\}}(x)=X \ \mbox{for every } x\in X.$
\end{rem}
\begin{comment}
\begin{defn}
Let $\{T_n\} \subseteq \mathcal{B}(X).$ A vector $x$ is called
periodic for $\{T_n\}$ if there exists positive integers $n$ such
that $T_n x=x.$
\end{defn}
\end{comment}
\begin{lem}\label{L1}
Let $x_n \rightarrow x$ and $y_n \rightarrow y$ for some $x, y \in
X$. If $y_k \in J_{\{T_n\}}(x_k) $ for every $k=1,2,3,...$, then
$y\in J_{\{T_n\}}(x)$.
\end{lem}
\begin{proof}
The proof is followed same as  the  single operator case stated in
\cite{MR2447911}. By our hypothesis there exists a positive integer
$k_1$ such that $\|x_{k_1}-x\|<\frac{1}{2}$ and
$\|y_{k_1}-y\|<\frac{1}{2}$. Since $y_{k_1} y \in
J_{\{T_n\}}(x_{k_1} )$ there exist a positive integer $l_1$ and $z_1
\in X$ such that $\|z_1- x_{k_1}\|<\frac{1}{2}$ and
$\|T_{l_1}z_1-y_{k_1}y\|<\frac{1}{2}$. Hence $\|z_1- x\|<1$ and
$\|T_{l_1}z_1-y\|<1$ which completes the firs step of the induction.
By proceeding inductively we now construct a strictly increasing
sequence of positive integers $\{l_n\}$ and a sequence of vectors
$\{l_n\}$ in $X$ such that $\|z_n- x\|<\frac{1}{n}$ and
$\|T_{l_n}z_n-y\|<\frac{1}{n}$. Therefore $y\in J_{\{T_n\}}(x)$.
\end{proof}
\begin{lem}
Let $\{T_n\} \subseteq \mathcal{B}(X)$ such that $T_n \rightarrow T$
uniformly on $X$ for some $T\in \mathcal{B}(X).$ If  $T_n y \in
J_{\{T_n\}}(T_n x) $ for every $n=1,2,3,...$ and for arbitrary $x, y
\in X$, then $Ty\in J_{\{T_n\}}(Tx)$.
\end{lem}
\begin{proof}
For an arbitrary $\epsilon >0$ there exists a positive integer $k_1$
such that $\|T_{k_1}-T\|<\frac{\epsilon}{2}$. Since $T_{k_1} y \in
J_{\{T_n\}}(T_{k_1} x)$ there exist a positive integer $l_1$ and
$z_1 \in X$ such that $\|z_1- T_{k_1}x\|<\frac{\epsilon}{2}$ and
$\|T_{l_1}z_1-T_{k_1}y\|<\frac{\epsilon}{2}$. Hence $\|z_1-
Tx\|<\epsilon$ and $\|T_{l_1}z_1-Ty\|<\epsilon$ i.e., $Ty\in
J_{\{T_n\}}(Tx)$.
\end{proof}
\begin{defn}
Let $\{T_n\}_{n=1}^{\infty}\subseteq \mathcal{B}(X)$  and $x\in X$.
The \emph{limit set} of $x$ under $\{T_n\}_{n=1}^{\infty}$, denoted
by $L_{\{T_n\}}(x)$ is defined as follows
\begin{eqnarray}
L_{\{T_n\}}(x)=\{y\in X: \mbox{there exist a strictly increasing
sequence of positive}\nonumber \\ \mbox{ integers}\ \{k_n\},\
T_{k_n}x \rightarrow
 y\}.\nonumber
\end{eqnarray}
In case $\{T_n\}_{n=1}^{\infty}$ is sequence of invertible linear
operators the sets $L_{\{T_n\}}^+(x)$, $J_{\{T_n\}}^+(x)$
($L_{\{T_n\}}^-(x)$, $J_{\{T_n\}}^-(x)$) denote the limit set and
extended limit set of $x$ under $\{T_n\}_{n=1}^{\infty}$
($\{T_n^{-1}\}_{n=1}^{\infty}$).
\end{defn}
\begin{lem}\label{L2}
Let $\{T_n\}$ be a sequence of mutually commuting operators in
$\mathcal{B}(X)$. For all $x\in X$, the sets $J_{\{T_n\}}(x)$ and
$L_{\{T_n\}}(x)$ are closed. Moreover for each $T_m\in
\mathcal{B}(X)$ we have  $T_m J_{\{T_n\}}(x)\subseteq
J_{\{T_n\}}(T_m x)$ and  $T_m L_{\{T_n\}}(x)\subseteq
L_{\{T_n\}}(T_m x)$.
\end{lem}
\begin{proof}
Take $y\in J_{\{T_n\}}(x)$. There exist a strictly increasing
sequence of positive integers $\{k_n\}$  and a sequence  $\{x_n\}$
in $ X$ such that  $x_n \rightarrow x$  and
 $T_{k_n}x_n \rightarrow  y$. Then $T_{k_n}T_m x_n = T_mT_{k_n}(x)\rightarrow  T_m
 y$ and so $T_m y \in J_{\{T_n\}}(T_m x)$ since $T_m x_n \rightarrow T_m
 x$.
\end{proof}
\begin{lem}
Let $\{T_n\}$ be a commuting sequence of invertible operators. Then
 $T_n^{-1} J_{\{T_n\}}(x)= J_{\{T_n\}}(T_n x)$ for every $x\in X.$
\end{lem}
\begin{proof}
By Lemma \ref{L2}, we have $ J_{\{T_n\}}(x)\subseteq T^{-1}_n
J_{\{T_n\}}(T_n x)$. Now pick $y\in T_n^{-1} J_{\{T_n\}}(x)$. There
are a strictly increasing sequence of positive integers $\{k_i\}$
and a sequence  $\{x_i\}$ in $ X$ such that  $x_i \rightarrow T_n x$
and
 $T_{k_i}x_i \rightarrow  T_n y$ as $i\rightarrow \infty$. Hence $T_n^{-1} x_i \rightarrow
 x$ and  $T_{k_i}T_n^{-1}x_i =T_n^{-1}T_{k_i} x_i \rightarrow   y$. That is $y\in J_{\{T_n\}}(x)$.
\end{proof}
\begin{prop}
Let $\{T_n\}$ be a sequence of invertible operators. Then $y\in
J_{\{T_n\}}^+(x)$ if and only if $x\in J_{\{T_n\}}^- (y)$ for every
$x, y \in X.$
\end{prop}
\begin{proof}
Choose $y\in J_{\{T_n\}}^+(x)$. There are a strictly increasing
sequence of positive integers $\{k_n\}$  and a sequence  $\{x_n\}$
in $ X$ such that  $x_n \rightarrow x$  and
 $T_{k_n}x_n \rightarrow  y$. Then $T_{k_n}^{-1}T_{k_n}x_n=x_n \rightarrow
 x$ and so $x\in J_{\{T_n\}}^- (y)$. The reverse implication is
 proved similarly.
\end{proof}
\begin{prop}
For $\{T_n\} \subseteq \mathcal{B}(X)$ assume that $\sup_n \|T_n\|<
\infty.$ Then $J_{\{T_n\}}(x)= L_{\{T_n\}}(x)$ for every $x \in X.$
\end{prop}
\begin{proof}
Note that the inclusion $L_{\{T_n\}}(x) \subseteq J_{\{T_n\}}(x)$
holds trivially. Now suppose $y\in J_{\{T_n\}}(x)$. Then there exist
a strictly increasing sequence of positive integers $\{k_n\}$  and a
sequence  $\{x_n\}$ in $ X$ such that  $x_n \rightarrow x$  and
 $T_{k_n}x_n \rightarrow  y$. Put $M=\sup_n \|T_n\|$. Thus we have
 \begin{eqnarray}
   \|T_{k_n}x-y\| &\leq &  \|T_{k_n}x-T_{k_n}x_n\|+ \|T_{k_n}x_n-y\|\nonumber \\
    &\leq & M \|x_n-x\|+ \|T_{k_n}x_n-y\|\rightarrow 0\nonumber
 \end{eqnarray}
 as $n\rightarrow \infty $. That is $ y\in L_{\{T_n\}}(x).$
\end{proof}
\begin{lem}\label{L3}
If $J_{\{T_n\}}(x)=X$ for some non-zero vector $x\in X$, then
$J_{\{T_n\}}(\lambda x)=X$ for every $\lambda \in \mathbb{C}.$
\end{lem}
\begin{proof}
The case $\lambda\neq 0$ is evident, since it is easy to check that
$J_{\{T_n\}}(\lambda x)= \lambda J_{\{T_n\}}( x)$. For the rest of
the proof let $\{\lambda_n\}$ be a sequence of non-zero complex
numbers converging to 0 and take $y\in J_{\{T_n\}}(x)$. Then by the
previous equation $y\in J_{\{T_n\}}(\lambda_n x)$ for every $n$.
Since $\lambda_n \rightarrow 0$, an application of Lemma \ref{L1}
yields that $y\in J_{\{T_n\}}(0)$. Therefore $J_{\{T_n\}}(0)=X$.
\end{proof}
\begin{prop}
Let $\{T_n\} \subseteq \mathcal{B}(X)$. Then the set $A=\{x\in X:
J_{\{T_n\}}(x)=X\}$ is a closed, connected and $T_n A \subseteq A.$
\end{prop}
\begin{proof}
By Lemma \ref{L1} it is easily inferred that $A$ is closed. To show
that $A$ is connected set let $x\in A$ be an arbitrary vector. Then
Lemma \ref{L3}, for every $\lambda \in \mathbb{C}$ implies that
$J_{\{T_n\}}(0)=J_{\{T_n\}}(\lambda x)= \lambda J_{\{T_n\}}(x)=X$.
Thus $A$ is connected.
\end{proof}
\begin{thm}\label{T1}
Let $\{T_n\} \subseteq \mathcal{B}(X)$. The following are
equivalent.
\begin{itemize}
  \item [(i)] The sequence $\{T_n\}$ is hypercyclic;
  \item [(ii)] The equation $J_{\{T_n\}}(x)=X$ holds for every $x\in
  X$;
  \item [(iii)] The set $A=\{x\in X: J_{\{T_n\}}(x)=X\}$ is dense in
  $X$;
  \item [(iv)] The set  $A=\{x\in X: J_{\{T_n\}}(x)=X\}$ has
  non-empty interior.
\end{itemize}
\end{thm}

\begin{proof}
To prove that $(i)$ implies $(ii)$ let $x, y \in X$. We know that
the set of hypercyclic vectors is $G_\delta$ and dense in $X$ (see
\cite{MR2826721}). Hence there exits a sequence of hypercyclic
vectors $\{x_n\}$ such that $x_n \rightarrow x$. Subsequently one
may find a strictly increasing sequence of positive integers
$\{k_n\}$  on which
 $T_{k_n}x_n \rightarrow  y$ as $n\rightarrow \infty$. Namely $y \in
 J_{\{T_n\}}(x)$.\\
 The implication $(ii)\Rightarrow (iii)$ is trivial. Moreover
 applying Lemma \ref{L1} makes sure that $(iii)$ implies $(ii)$. \\
 Now we show that $(iv)$ implies $(ii)$. Fix $x\in A^\circ$ (interior of
 $A$) and let $y\in X$ be an arbitrary. Since $y\in
 X=J_{\{T_n\}}(x)$ there exist
a strictly increasing sequence of positive integers $\{k_n\}$  and a
sequence  $\{x_n\}$ in $ X$ such that  $x_n \rightarrow x$  and
 $T_{k_n}x_n \rightarrow  y$.
 Without loss of generality we may assume that $x_n\in A$ for every
 $n\in \mathbb{N}$. Indeed this can be done in sake of the our
 assumption $x\in A^\circ$. \\
 To prove the implication $(ii)\Rightarrow (i)$, let  $\{B_j\}$ be a
 countable open basis for the relative topology of $X$.  By the
 Baire category theorem and the hypercyclic vectors description (\cite{MR2826721}, \cite{MR2720700})
 it is sufficient to show that the set $\bigcup_{n=1}^{\infty} T_n^{-1}( B_j)$ is
  dense in $X$ for every $j$. In fact the set of all
 hypercyclic vectors for $\{T_n\}$ is equal
 $$\bigcap_{j=1}^{\infty}
\bigcup_{n=1}^{\infty} T_n^{-1}( B_j).$$ At this moment consider the
equivalent definition of $J_{\{T_n\}}(x)$ described in Remark
\ref{R1}. Let $x\in X$,  the neighborhood $U$ of $x$ and $B_j$ be
given. Since $J_{\{T_n\}}(x)=X$, there exists $y\in U$ and $n\in
\mathbb{N}$ such that $T_n y \in B_j$ or equivalently $y\in
T_n^{-1}( B_j)$. This fact completes the proof since $U$ was chosen
arbitrarily.
\end{proof}
%\begin{comment}
\begin{prop}
Let $\{T_n\} \subseteq \mathcal{B}(X)$. Then the set
$A=\bigcup_{n=1}^{\infty}\{x\in X: J_{{T_n}}(x)=X\}$ is a closed,
connected and $\bigcup_{n=1}^{\infty}T_n A \subseteq A.$
\end{prop}
\begin{proof}
Let $x\in A$. There exists $n_0\in \mathbb{N}$ such that
$J_{{T_{n_0}}}(x)=X$. The $T_{n_0}$-invariance of $J_{T_{n_0}}(x)$
implies that $J_{{T_{n_0}}}(T_{n_0}x)=X$. Then $T_{n_0}x\in A$.
\end{proof}
%%%%%%%%%%%%%%%%%%%%%%%%%%%%%%%%%%
\begin{comment}
\begin{thm}
Let $\{T_n\} \subseteq \mathcal{B}(X)$. The following are
equivalent.
\begin{itemize}
  \item [(i)] The sequence $\{T_n\}$ is hypercyclic;
  \item [(ii)] The equation $J_{T_n}(x)=X$ holds for every $x\in
  X$ and $T_n\in \mathcal{B}(X) $;
  \item [(iii)] The set $A=\bigcup_{n=1}^{\infty}\{x\in X: J_{{T_n}}(x)=X\}$ is dense in
  $X$;
  \item [(iv)] The set  $A=\bigcup_{n=1}^{\infty}\{x\in X: J_{{T_n}}(x)=X\}$ has
  non-empty interior.
\end{itemize}
\end{thm}
\begin{proof}
\end{proof}
\end{comment}
%%%%%%%%%%%%%%%%%%%%%%%%%%%%%%%%%%%%%
\begin{lem}
Let $\{T_n\}$ be a mutually  commuting subalgebra of $
\mathcal{B}(X)$. Then \break $\bigcup_{n=1}^{\infty}T_i
J_{\{T_n\}}(x)\subseteq J_{\{T_n\}}(T_{n_0}x)$ for some
$T_{\{n_0\}}\in \mathcal{B}(X)$.
\end{lem}
\begin{proof}
Let $y' \in \bigcup_{n=1}^{\infty}T_i J_{\{T_n\}}(x)$. Then
$y'=T_{n_0}y$ for some  $T_{n_0}\in \mathcal{B}(X) $ and $y\in
J_{\{T_n\}}(x)$. So there exist a strictly increasing sequence of
positive integers $\{k_n\}$  and a sequence  $\{x_n\}$ in $ X$ such
that  $x_n \rightarrow x$  and  $T_{k_n}x_n \rightarrow  y$. Hence
$T_{n_0}T_{k_n}x_n \rightarrow  T_{n_0}y$. Indeed $T_{k_n}T_{n_0}x_n
\rightarrow y'$ i.e., $y' \in J_{\{T_n\}}(T_{n_0}x)$ since
$T_{n_0}x_n \rightarrow T_{n_0}x$.
\end{proof}
\begin{prop}\label{P1}
Let $\{T_n\} \subseteq \mathcal{B}(X)$ be mutually commuting.
Suppose there exists a cyclic vector $x\in X$ for $\{T_n\}$ such
that $J_{\{T_n\}}(x)$ has nonempty interior. Then for every nonzero
polynomial $P$ and $n\in \mathbb{N}$  the operators $P(T_n)$ has
dense range. Moreover the point spectrum $\sigma_p (T^*_n)$ of
 the adjoint operator of $T_n$ is empty.
\end{prop}
\begin{proof}
First assume that $X$ is a complex Banach space. For each $n\in
\mathbb{N}$ write $P(T_n)= \alpha(T_n - \lambda_1 I)(T_n - \lambda_2
I)...(T_n - \lambda_k I)$ for some $\alpha, \lambda_i \in \mathbb{C}
(i=1,2,...,k)$ where $I$ stands for the identity operator. So it
suffices to show that $T_n - \lambda I$ has dense range for every
$\lambda \in \mathbb{C}$. Suppose on contrary there exists a nonzero
linear functional $x^*$ such that $x^*(T_n - \lambda I)(x)=0$ for
every $x\in X.$ Taking  $y\in J_{\{T_n\}}^\circ(x)$ is caused to
find  a strictly increasing sequence of positive integers $\{k_n\}$
and a sequence $\{x_n\} \subseteq X$  such that  $x_n \rightarrow x$
and $T_{k_n}x_n \rightarrow  y$. Now by passing to the subsequence
$\{k_n\}$ and letting $n \rightarrow \infty$ we have $x^*(y)=\lambda
x^*(x)$ which is contradiction.
\end{proof}
\begin{thm}
Let $\{T_n\} \subseteq \mathcal{B}(X)$ be mutually commuting. Then
$\{T_n\}$ is hypercyclic if and only if there exists a cyclic vector
$x\in X$ for $\{T_n\}$ such that $J_{\{T_n\}}(x)=X$.
\end{thm}
\begin{proof}
We just prove the only if part of assertion. Let $x\in X$
 be acyclic vector for $\{T_n\}$,  $J_{\{T_n\}}(x)=X$ and let
 $P$ be an arbitrary nonzero polynomial. By Lemma \ref{L2} we have
 $P(T_m) J_{\{T_n\}}(x)\subseteq
J_{\{T_n\}}(P(T_m) x)$ for every $m\in \mathbb{N}$. By previous
proposition every $P(T_m)$ has dense range. By applying this fact
and Lemma \ref{L2} to the above inclusion we obtain
$$X= \overline{P(T_m)X}\subseteq
J_{\{T_n\}}(P(T_m) x).$$ Thus $J_{\{T_n\}}(P(T_m) x)=X.$ Now using
the fact that $x$ is a cyclic vector follows that there exists a
dense subset $Y$ of $X$ on which $J_{\{T_n\}}(y)=X$ for every $y\in
Y$. Therefore by Theorem \ref{T1}, $\{T_n\}$ is hypercyclic.
\end{proof}
The following corollary can be readily deduced by the same technique
used in the proof of Proposition \ref{P1}.
\begin{cor}
Let $\{T_n\} \subseteq \mathcal{B}(X)$. Suppose there exists a
vector $x\in X$ such that  $J_{\{T_n\}}(x)$ has nonempty interior.
Then for every sequence of complex numbers $\{\lambda_n\}$ with
$\lambda_n \rightarrow 0$ the sequence of bounded linear  operators
$\{T_n - \lambda_n I\}$ has dense range.
\end{cor}

%%%%%%%%%%%%%%%%%%%%%%%%%%%%%%%%%%%%%%%%%%%%%%%%%%%%%%%%%%%%%%%%%%%%%%%%%%%%
%%%%%%%%%%%%%%%%%%%%%%%%%%%%%%%%%%%%%%%%%%%%%%%%%%%%%%%%%%%%%%%%%%%%%%%%%%%%

%%%%%%%%%%%%%%%%%%%%%%%%%%%%%%%%%%%%%%%%%%%%%%%%%%%%%%%%%%%%%%%%%%%%%%%%%%%%%%%
%\bibliography{mybibfile}

\end{document}